\documentclass[]{scrartcl}

\usepackage[]{amsmath, amssymb,amsfonts,amsthm,mathtools,braket,color,enumerate}

\usepackage[margin=25mm]{geometry}

\usepackage{hhline}
\usepackage{url}
\usepackage{here}
\usepackage{stmaryrd}
\usepackage{zref-clever}
\zcsetup{cap=true}
\AddToHook{env/proposition/begin}{\zcsetup{countertype={theorem=proposition}}}
\AddToHook{env/corollary/begin}{\zcsetup{countertype={theorem=corollary}}}
\AddToHook{env/lemma/begin}{\zcsetup{countertype={theorem=lemma}}}
\AddToHook{env/example/begin}{\zcsetup{countertype={theorem=example}}}

\usepackage{hyperref}
\hypersetup{
  colorlinks   = true, 
  urlcolor     = blue, 
  linkcolor    = blue, 
  citecolor   = red 
}

\usepackage{tikz-cd}
\usepackage{tikz}
\usetikzlibrary{decorations}
\usetikzlibrary{arrows}
\tikzstyle{v} = [circle, draw, inner sep=2pt, minimum size=3pt, fill=black]
\tikzstyle{l} = [rectangle, draw, rounded corners]

\theoremstyle{plain}
\newtheorem{theorem}{Theorem}[section]
\newtheorem{lemma}[theorem]{Lemma}
\newtheorem{proposition}[theorem]{Proposition}
\newtheorem{corollary}[theorem]{Corollary}

\theoremstyle{definition}
\newtheorem{definition}[theorem]{Definition}

\newtheorem{example}[theorem]{Example}

\newtheorem{question}[theorem]{Question}
\newtheorem{remark}[theorem]{Remark}

\DeclareMathOperator{\codim}{codim}

\DeclareMathOperator{\wt}{wt}
\DeclareMathOperator{\Kernel}{Ker}
\DeclareMathOperator{\Image}{Im}
\DeclareMathOperator{\Mat}{M}
\DeclareMathOperator{\GL}{GL}
\DeclareMathOperator{\Hom}{Hom}
\DeclareMathOperator{\rank}{rank}

\title {The Coboundary Quasi-Polynomials of Hyperplane Arrangements
over Residually Finite Dedekind Domains}
\author{
Masamichi Kuroda
\thanks{Faculty of Engineering, Nippon Bunri University, Oita 870-0316, Japan. \\ 
E-mail:kurodamm@nbu.ac.jp}
\and
Norihiro Nakashima
\thanks{Department of Mathematics, Nagoya Institute of Technology, Aichi, 466-8555, Japan. \\
E-mail:nakashima@nitech.ac.jp}
\and
Shuhei Tsujie
\thanks{Department of Mathematics, Hokkaido University of Education, Asahikawa, Hokkaido 070-8621, Japan. \\
E-mail:tsujie.shuhei@a.hokkyodai.ac.jp}
}

\date{}

\begin{document}
\maketitle
	
\begin{abstract}
The characteristic polynomial plays an important role in study of hyperplane arrangements. 
There are several refinements of the characteristic polynomial. 
One of them is the coboundary polynomial defined by Crapo. 
Another refinement is the characteristic quasi-polynomial for an integral arrangement defined by Kamiya, Takemura, and Terao. 
Recently, the first and third authors introduced the characteristic quasi-polynomial for arrangement defined over a residually finite Dedekind domain. 
In this article, we introduce the common refinement of the coboundary polynomial and characteristic quasi-polynomial for an arrangement over a residually finite Dedekind domain. 
\end{abstract}



\tableofcontents

\section{Introduction}

Let $F$ be a field and $\ell$ a positive integer. 
A (central) \textbf{hyperplane arrangement} $\mathcal{A}$ is a finite collection of linear hyperplanes in a vector space $F^{\ell}$. 
The \textbf{characteristic polynomial} of $\mathcal{A}$ is defined by the combinatorial information of the intersections of the hyperplanes in $\mathcal{A}$. 

The \textbf{coboundary polynomial} introduced by Crapo \cite{crapo1969tutte-am} is an important refinement of the characteristic polynomial, which is equivalent to the \textbf{Tutte polynomial}. 
For an arrangement over a finite field, the coboundary polynomial is also equivalent to the \textbf{weight enumerator} of the corresponding linear code. 

For an arrangement over integers, Kamiya, Takemura, and Terao \cite{kamiya2008periodicity-joac} introduced the \textbf{characteristic quasi-polynomial}, which is another refinement of the characteristic polynomial. 
Recently, the first and third authors generalized the characteristic quasi-polynomial for an arrangement over a residually finite Dedekind domain, that is,  a Dedekind domain in which every residue ring modulo a nonzero ideal is finite.

The main purpose of this article is to introduce the \textbf{coboundary quasi-polynomial} based on the relation between the coboundary polynomial and the weight enumerator.
We will prove that the coboundary quasi-polynomial is the common refinement of the coboundary polynomial and the characteristic quasi-polynomial. 
Let $\mathcal{O}$ denote a residually finite Dedekind domain and $K$ its field of fractions. 
Suppose that $\mathcal{A}$ is a finite subset of $\mathcal{O}^{\ell}$. 
Then the relations of the (quasi-)polynomials are summarized as follows. 

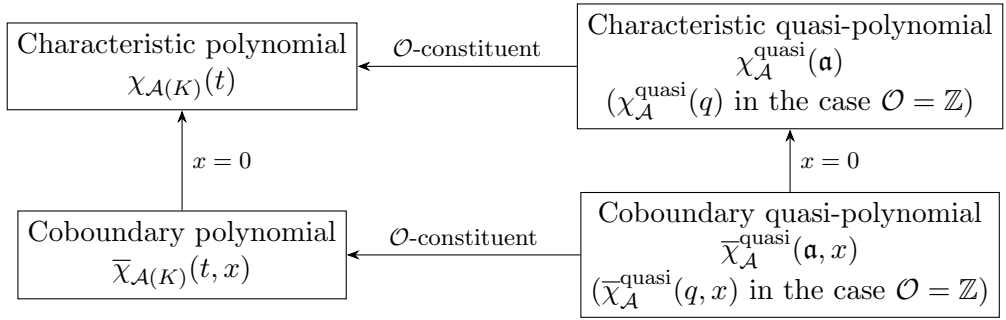
\begin{figure}[h]
\centering
\begin{tikzpicture}
\node[align=center, draw, rectangle] (cp) at (0,2.5) {Characteristic polynomial \\ $\chi_{\mathcal{A}(K)}(t)$};
\node[align=center, draw, rectangle] (cqp) at (8,2.5) {Characteristic quasi-polynomial \\ $\chi_{\mathcal{A}}^{\mathrm{quasi}}(\mathfrak{a})$ \\ ($\chi_{\mathcal{A}}^{\mathrm{quasi}}(q)$ in the case $\mathcal{O}=\mathbb{Z}$)};
\node[align=center, draw, rectangle] (cbp) at (0,0) {Coboundary polynomial \\ $\overline{\chi}_{\mathcal{A}(K)}(t,x)$};
\node[align=center, draw, rectangle] (cbqp) at (8,0) {Coboundary quasi-polynomial \\ $\overline{\chi}_{\mathcal{A}}^{\mathrm{quasi}}(\mathfrak{a},x)$ \\ ($\overline{\chi}_{\mathcal{A}}^{\mathrm{quasi}}(q,x)$ in the case $\mathcal{O}=\mathbb{Z}$)};
\draw[-Stealth] (cqp) -- node[anchor=south, align=center, font=\footnotesize]{$\mathcal{O}$-constituent}(cp);
\draw[-Stealth] (cbqp) -- node[anchor=south, align=center, font=\footnotesize]{$\mathcal{O}$-constituent}(cbp);
\draw[-Stealth] (cbqp) -- node[anchor=west, font=\footnotesize]{$x=0$} (cqp);
\draw[-Stealth] (cbp) -- node[anchor=west, font=\footnotesize]{$x=0$} (cp);
\end{tikzpicture}
\end{figure}

We also prove that every constituent of the coboundary quasi-polynomial has a combinatorial interpretation as every constituent of the characteristic quasi-polynomial does. 
The results explained above are stated in \zcref{main theorem}.

\subsection{Characteristic polynomials, coboundary polynomials, and weight enumerators}

Let $\mathcal{A}$ be a hyperplane arrangement in the $\ell$-dimensional vector space $F^{\ell}$ over a field $F$. 
Define the \textbf{intersection lattice} $L(\mathcal{A})$ by 
\begin{align*}
L(\mathcal{A}) \coloneqq \Set{\bigcap_{H \in \mathcal{B}}H | \mathcal{B} \subseteq \mathcal{A}}
\end{align*}
which is partially ordered by reverse inclusion. 
Note that when $\mathcal{B}$ is the emptyset we regard the intersection as the ambient space $\hat{0} \coloneqq F^{\ell}$, which is the minimum element of $L(\mathcal{A})$.  

Define the \textbf{characteristic polynomial} $\chi_{\mathcal{A}}(t)$ and the \textbf{coboundary polynomial} $\overline{\chi}_{\mathcal{A}}(t,x)$ by 
\begin{align*}
\chi_{\mathcal{A}}(t) \coloneqq \sum_{X \in L(\mathcal{A})}\mu(\hat{0}, X)t^{\dim X} \quad \text{ and } \quad 
\overline{\chi}_{\mathcal{A}}(t,x) \coloneqq \sum_{X \in L(\mathcal{A})} \left(\sum_{Y \geq X}\mu(X,Y)t^{\dim Y}\right)x^{a(X)}, 
\end{align*}
where $\mu$ is the \textbf{M\"obius function} on $L(\mathcal{A})$ defined recursively by 
\begin{align*}
\mu(X,Y) \coloneqq \begin{cases}
1 & \text{ if } Y=X; \\
-\sum_{X \leq Z < Y}\mu(X,Z) & \text{ if } Y<X; \\
0 & \text{ otherwise,}
\end{cases}
\end{align*}
and $a(X)$ denotes the number of hyperplanes in $\mathcal{A}$ containing $X$. 
Note that the coboundary polynomial is a refinement of the characteristic polynomial since $\overline{\chi}_{\mathcal{A}}(t,0) = \chi_{\mathcal{A}}(t)$ holds. 

The \textbf{Tutte polynomial} $T_{\mathcal{A}}(x,y)$ of $\mathcal{A}$ is defined by 
\begin{align*}
T_{\mathcal{A}}(x,y) = \sum_{\mathcal{B} \subseteq \mathcal{A}}(x-1)^{r(\mathcal{A})-r(\mathcal{B})}(y-1)^{\#\mathcal{B}-r(\mathcal{B})},
\end{align*}
where $r(\mathcal{B}) \coloneqq \codim \bigcap_{H \in \mathcal{B}}H$, which coincides with the Tutte polynomial of the linear matroid defined by $\mathcal{A}$. 
The coboundary polynomial and the Tutte polynomial are related by 
\begin{align*}
\overline{\chi}_{\mathcal{A}}((x-1)(y-1), y) = (x-1)^{\ell-r(\mathcal{A})}(y-1)^{\ell}T_{\mathcal{A}}(x,y). 
\end{align*}
Therefore they are equivalent notion. 
See \cite{ardila2007computing-pjom} for the study of the Tutte polynomials and the coboundary polynomials of arrangements. 

In the case where $F$ is a finite field $\mathbb{F}_{p^{m}}$, it is well known that $\chi_{\mathcal{A}}(p^{m})$ is the cardinality of the complement of $\mathcal{A}$, that is, 
\begin{align*}
\chi_{\mathcal{A}}(p^{m}) = \#\left(\mathbb{F}_{p^{m}}^{\ell}\setminus\bigcup_{H \in \mathcal{A}}H\right). 
\end{align*}
The coboundary polynomial $\overline{\chi}_{\mathcal{A}}(t,x)$ has a similar property as follows (\zcref{Greene}).

Let $C$ be a linear code (i.e., a vector subspace) in $\mathbb{F}_{p^{m}}^{n}$. 
Define the \textbf{weight enumerator} $W_{C}(x,y)$ of $C$ by 
\begin{align*}
W_{C}(x,y) \coloneqq \sum_{u \in C}x^{n-\wt(u)}y^{\wt(u)}, 
\end{align*}
where $\wt(u)$ denotes the \textbf{Hamming weight} of a codeword $u \in C$.

Let $\ell \coloneqq \dim C$. 
Then there exists a matrix $G = (\alpha_{1} \cdots \alpha_{n}) \in \Mat_{\ell \times n}(\mathbb{F}_{p^{m}})$, called a \textbf{generator matrix} of $C$, such that $\Image G = C, \, \Kernel G = \{0\}$,  and $\alpha_{i} \neq k\alpha_{j}$ for any $k \in \mathbb{F}_{p^{m}}$ and distinct $i,j \in \{1, \dots, n\}$, where we identify the matrix $G$ with a map from $\mathbb{F}_{p^{m}}^{\ell}$ to $\mathbb{F}_{p^{m}}^{n}$ defined by the right multiplication $u \mapsto uG$. 
Define the arrangement $\mathcal{A}_{G}$ by $\mathcal{A}_{G} = \{H_{\alpha_{1}}, \dots, H_{\alpha_{n}}\}$, where 
\begin{align*}
H_{\alpha_{i}} \coloneqq \Set{u \in \mathbb{F}_{p^{m}}^{\ell} | u\alpha_{i} = 0}. 
\end{align*}

\begin{theorem}[\cite{greene1976weight-siam}]\label{Greene}
With the notation above, 
\begin{align*}
\overline{\chi}_{\mathcal{A}_{G}}(p^{m},x) = W_{C}(x,1). 
\end{align*}
\end{theorem}

Based on the equality in \zcref{Greene}, we will define the coboundary quasi-polynomial, which can be regarded as a generalization of weight enumerator. 
Jurrius and Pellikaan \cite[Section 1.5.2]{jurrius2013codes-soctac} studied a similar but distinct generalization of the weight enumerator, which is called the extended weight enumerator.

\subsection{Characteristic quasi-polynomials of integral arrangements}

There is another refinement of the characteristic polynomial, which is called the \textbf{characteristic quasi-polynomial}. 
Let $\mathcal{A} = \{\alpha_{1}, \dots, \alpha_{n}\}$ be a finite subset of nonzero vectors in $\mathbb{Z}^{\ell}$ and define the hyperplane $H_{\alpha_{i}}$ in $\mathbb{Q}^{\ell}$ by 
\begin{align*}
H_{\alpha_{i}, \mathbb{Q}} \coloneqq \Set{u \in \mathbb{Q}^{\ell} | u\alpha_{i} = 0}. 
\end{align*}
Then $\mathcal{A}(\mathbb{Q}) \coloneqq \{H_{\alpha_{1}, \mathbb{Q}}, \dots, H_{\alpha_{n}, \mathbb{Q}}\}$ is a hyperplane arrangement over $\mathbb{Q}$. 
Since every vector in $\mathcal{A}$ is integral, for every positive integer $q$, we can consider the arrangement $\mathcal{A}(\mathbb{Z}/q\mathbb{Z}) \coloneqq \{H_{\alpha_{1}, \mathbb{Z}/q\mathbb{Z}}, \dots, H_{\alpha_{n}, \mathbb{Z}/q\mathbb{Z}}\}$ in $\left(\mathbb{Z}/q\mathbb{Z}\right)^{\ell}$ by taking modulo $q$, where 
\begin{align*}
H_{\alpha_{i}, \mathbb{Z}/q\mathbb{Z}} \coloneqq \Set{[u]_{q} \in \left(\mathbb{Z}/q\mathbb{Z}\right)^{\ell} | [u]_{q}\alpha_{i} = 0}
\end{align*}
and $[u]_{q}$ stands for the equivalence class of $u \in \mathbb{Z}^{\ell}$.

Kamiya, Takemura, and Terao introduced the map $\chi_{\mathcal{A}}^{\mathrm{quasi}} \colon \mathbb{Z}_{>0} \to \mathbb{Z}$ defined by 
\begin{align*}
\chi_{\mathcal{A}}^{\mathrm{quasi}}(q) \coloneqq \# \left(\left(\mathbb{Z}/q\mathbb{Z}\right)^{\ell} \setminus \bigcup_{i=1}^{\ell} H_{\alpha_{i}, \mathbb{Z}/q\mathbb{Z}}\right),  
\end{align*}
which is called the \textbf{characteristic quasi-polynomial} of $\mathcal{A}$ because of the following theorem. 

\begin{theorem}[{\cite[Theorem 2.4 and Theorem 2.5]{kamiya2008periodicity-joac}}]\label{KTT}
The function $\chi_{\mathcal{A}}^{\mathrm{quasi}}(q)$ is a quasi-polynomial with GCD-property and its first constituent coincides with the characteristic polynomial of the hyperplane arrangement $\mathcal{A}(\mathbb{Q})$. 
Namely, there exist a positive integer $\rho_{\mathcal{A}}$ and polynomial $\chi_{\mathcal{A}}^{k}(t) \in \mathbb{Z}[t]$ for every divisor $k$ of $\rho_{\mathcal{A}}$ such that $\gcd(q,\rho_{\mathcal{A}}) = k$ implies $\chi_{\mathcal{A}}^{\mathrm{quasi}}(q) = \chi_{\mathcal{A}}^{k}(q)$. 
Moreover, $\chi_{\mathcal{A}}^{1}(t) = \chi_{\mathcal{A}(\mathbb{Q})}(t)$. 
\end{theorem}

The integer $\rho_{\mathcal{A}}$ is called the \textbf{LCM-period} (See \cite[p.~323]{kamiya2008periodicity-joac} for the definition). 
We call the polynomial $\chi_{\mathcal{A}}^{k}(t)$ the \textbf{$k$-constituent} of the characteristic quasi-polynomial $\chi_{\mathcal{A}}^{\mathrm{quasi}}(q)$. 

Every nonzero vector $a \in \mathbb{Z}^{\ell}$ defines the hypertori $H_{a, \mathbb{C}^{\times}}$ in the algebraic torus $\left(\mathbb{C}^{\times}\right)^{\ell}$ by 
\begin{align*}
H_{a, \mathbb{C}^{\times}} \coloneqq \Set{u = (u_{1}, \dots, u_{\ell}) \in \left(\mathbb{C}^{\times}\right)^{\ell} | u_{1}^{a_{1}} \cdots u_{\ell}^{a_{\ell}} = 1}. 
\end{align*}
We call the collection $\mathcal{A}(\mathbb{C}^{\times}) \coloneqq \{H_{\alpha_{1}, \mathbb{C}^{\times}}, \dots, H_{\alpha_{n}, \mathbb{C}^{\times}}\}$ the \textbf{toric arrangement} defined by $\mathcal{A}$. 
A connected component of the intersection of some members in $\mathcal{A}(\mathbb{C}^{\times})$ is called a \textbf{layer}. 
The layers form a poset $L(\mathcal{A}(\mathbb{C}^{\times}))$ ordered by reverse inclusion. 
For a positive integer $k$, let $L(\mathcal{A}(\mathbb{C}^{\times}))[k]$ denote the subposet of $L(\mathcal{A}(\mathbb{C}^{\times}))$ consisting of layers which contain a $k$-torsion element. 
Liu, Tran, and Yoshinaga proved the following theorem. 

\begin{theorem}[{\cite[Corollary 5.6]{liu2021$g$-tutte-imrn}} and {\cite[Corollary 4.8]{tran2019combinatorics-joctsa}}]\label{LTY}
Every constituent $\chi_{\mathcal{A}}^{k}(t)$ of the characteristic quasi-polynomial $\chi_{\mathcal{A}}^{\mathrm{quasi}}(q)$ coincides with the characteristic polynomial of $L(\mathcal{A}(\mathbb{C}^{\times}))[k]$. 
\end{theorem}

Using \zcref{LTY}, Higashitani, Tran, and Yoshinaga proved the following theorem. 

\begin{theorem}[{\cite[Theorem 1.2]{higashitani2022period-imrn}}]\label{HTY}
The LCM-period $\rho_{\mathcal{A}}$ is the minimum period for $\chi_{\mathcal{A}}^{\mathrm{quasi}}(q)$. 
\end{theorem}

\begin{remark}
In this article, we consider only central arrangements. 
Non-central cases are studied in \cite{kamiya2011periodicity-aoc}. 
In non-central cases, LCM-period may not be minimum.
We call this phenomenon \textbf{period collapse}.  
See \cite{higashitani2025characteristic, higashitani2025characteristic-1, higashitani2022period-imrn} for studies about period collapse. 
\end{remark}

\subsection{Characteristic quasi-polynomials of arrangements over residually finite Dedekind domains}

A Dedekind domain is called \textbf{residually finite} if the residue ring is finite for every nonzero ideal. 
For example, it is widely known that the ring of integers of an algebraic field is a residually finite Dedekind domain. 

Let $\mathcal{O}$ be a residually finite Dedekind domain and $I(\mathcal{O})$ the set of nonzero ideals. 
Let $K$ be the field of fractions of $\mathcal{O}$. 
Let $\mathcal{A} = \{\alpha_{1}, \dots, \alpha_{n}\}$ be a finite set consisting of nonzero vectors in $\mathcal{O}^{\ell}$. 
Given an $\mathcal{O}$-module $M$, define the \textbf{$M$-plexification} $\mathcal{A}(M)$ by $\mathcal{A}(M) \coloneqq \{H_{\alpha_{1}, M}, \dots, H_{\alpha_{n}, M}\}$, where for each $i \in \{1, \dots, n\}$, 
\begin{align*}
H_{\alpha_{i}, M} \coloneqq \Set{u \in M^{\ell} | u \alpha_{i} = 0}. 
\end{align*}
For a nonempty subset $J \subseteq \mathcal{A}$, put 
\begin{align*}
H_{J, M} \coloneq \bigcap_{\alpha \in J}H_{\alpha, M} \quad \text{and} \quad H_{\varnothing, M} \coloneqq M^{\ell}.  
\end{align*}

Let $\mathfrak{a} \in I(\mathcal{O})$. 
An element $m \in M$ is called an \textbf{$\mathfrak{a}$-torsion element} if $\mathfrak{a}m \coloneqq \Set{am \in M | a \in \mathfrak{a}} = \{0\}$.
For a subset $S \subseteq M$, let $S[\mathfrak{a}]$ denote the subset of $S$ consisting of $\mathfrak{a}$-torsion elements.

Let $R$ be a commutative ring. 
A \textbf{quasi-polynomial on $I(\mathcal{O})$} is a function $f \colon I(\mathcal{O}) \to R$ satisfying the following statement: 
There exist an ideal $\rho \in I(\mathcal{O})$ and polynomials $f^{\kappa}(t) \in R[t]$ for each $\kappa \mid \rho$ such that for any $\mathfrak{a} \in I(\mathcal{O})$, $\mathfrak{a}+\rho = \kappa$ implies $f(\mathfrak{a}) = f^{\kappa}(N(\mathfrak{a}))$, where $N(\mathfrak{a}) \coloneqq \# \left(\mathcal{O}/\mathfrak{a}\right)$ denotes the absolute norm of $\mathfrak{a}$. 
The ideal $\rho$ is called a \textbf{period} of $f$ and each $f^{\kappa}(t)$ is called a \textbf{$\kappa$-constituent} of $f$. 
A period that divides every period is called the \textbf{minimum period} of $f$. 
Note that if $f_{1}$ and $f_{2}$ are quasi-polynomials on $I(\mathcal{O})$ with period $\rho_{1}$ and $\rho_{2}$, then the sum $f_{1}+f_{2}$ is a quasi-polynomial on $I(\mathcal{O})$ with period $\mathrm{lcm}(\rho_{1}, \rho_{2})$ (See \cite[Proposition 2.2]{kuroda2024characteristic-ecaa}). 

Define the function $\chi_{\mathcal{A}}^{\mathrm{quasi}} \colon I(\mathcal{O}) \to \mathbb{Z}$ by 
\begin{align*}
\chi_{\mathcal{A}}^{\mathrm{quasi}}(\mathfrak{a}) \coloneqq \#\left(\left(\mathcal{O}/\mathfrak{a}\right)^{\ell}\setminus\bigcup_{i=1}^{n}H_{\alpha_{i},\mathcal{O}/\mathfrak{a}}\right). 
\end{align*}
We call $\chi_{\mathcal{A}}^{\mathrm{quasi}}$ the \textbf{characteristic quasi-polynomial} of $\mathcal{A}$ since the following theorem holds, which is a generalization of \zcref{KTT} and \zcref{HTY}. 
\begin{theorem}[{\cite[Theorem 3.1 and Theorem 5.1]{kuroda2024characteristic-ecaa}}]\label{charateristic quasi-polynomial}
The following statements hold. 
\begin{enumerate}[(1)]
\item The function $\chi_{\mathcal{A}}^{\mathrm{quasi}}$ is a quasi-polynomial on $I(\mathcal{O})$ with the \textbf{LCM-period} $\rho_{\mathcal{A}}$, where 
\begin{align*}
\rho_{\mathcal{A}} \coloneqq \operatorname{lcm}\Set{\operatorname{Ann}(\operatorname{tors}(\operatorname{Coker}G_{J})) | \varnothing \neq J \subseteq \mathcal{A} }
\end{align*}
and $G_{J}$ denotes the map from $\mathcal{O}^{\ell}$ to $\mathcal{O}^{J}$ defined by $u \mapsto (u\alpha)_{\alpha \in J}$. 
\item The LCM-period $\rho_{\mathcal{A}}$ is the minimum period of $\chi_{\mathcal{A}}^{\mathrm{quasi}}$.  
\item The $\mathcal{O}$-constituent of $\chi_{\mathcal{A}}^{\mathrm{quasi}}$ coincides with the characteristic polynomial of $\mathcal{A}(K)$. 
\end{enumerate}
\end{theorem}

Let $\pi \colon K^{\ell} \to \left(K/\mathcal{O}\right)^{\ell}$ be the canonical projection. 
For any subset $J \subseteq \mathcal{A}$, $\pi(H_{J,K})$ is an $\mathcal{O}$-submodule of $H_{J, K/\mathcal{O}}$. 
We call an element of $H_{J, K/\mathcal{O}}/\pi(H_{J,K})$ a \textbf{layer} of $\mathcal{A}(K/\mathcal{O})$. 
Each layer behaves like a ``connected component" of $H_{J, K/\mathcal{O}}$ even if $K/\mathcal{O}$ does not have reasonable topology. 
Let $L(\mathcal{A}(K/\mathcal{O}))$ be the set consisting of all layers of $\mathcal{A}(K/\mathcal{O})$. 
The set $L(\mathcal{A}(K/\mathcal{O}))$ is a finite poset with the reverse inclusion as subsets in $\left(K/\mathcal{O}\right)^{\ell}$ and is called the \textbf{poset of layers} of $\mathcal{A}(K/\mathcal{O})$. 
Note that when $\mathcal{A}$ is defined over $\mathbb{Z}$, the poset of layers $L(\mathcal{A}(\mathbb{C}^{\times}))$ of the toric arrangement $\mathcal{A}(\mathbb{C}^{\times})$ is naturally isomorphic to $L(\mathcal{A}(\mathbb{Q}/\mathbb{Z}))$. 
In this sense, the $K/\mathcal{O}$-plexification $\mathcal{A}(K/\mathcal{O})$ can be regard as an analogy of a toric arrangement.

A layer $Z \subseteq \left(K/\mathcal{O}\right)^{\ell}$ may belong to two residue modules $H_{J_{1},K/\mathcal{O}}/\pi(H_{J_{1}, K})$ and $H_{J_{2},K/\mathcal{O}}/\pi(H_{J_{2}, K})$. 
In this case, $H_{J_{1}, K} = H_{J_{2}, K}$ holds (See \zcref{KT Lem 4.3}). 
Thus there exists a unique intersection $H_{Z} \in L(\mathcal{A}(K))$ such that $Z = \pi(u) + \pi(H_{Z})$ for any $\pi(u) \in Z$. 
We define the \textbf{dimension} of $Z$ as $\dim Z \coloneqq \dim_{K} H_{Z}$. 

A layer $Z = \pi(u)+\pi(H_{J, K}) \in H_{J,K/\mathcal{O}}/\pi(H_{J, K})$ is $\mathfrak{a}$-torsion element if and only if $\pi(au) \in \pi(H_{J, K}) = \pi(H_{Z})$ for any $a \in \mathfrak{a}$. 
Hence whether $Z$ is an $\mathfrak{a}$-torsion element is independent of the choice of $J$. 
As we will see later (\zcref{torsion lemma}), $Z$ is an $\mathfrak{a}$-torsion element if and only if $Z$ contains an $\mathfrak{a}$-torsion element in $\left(K/\mathcal{O}\right)^{\ell}$. 
The \textbf{$\mathfrak{a}$-torsion subposet} $L(\mathcal{A} (K / \mathcal{O}) )[\mathfrak{a}]$ is the subposet of $L(\mathcal{A} (K / \mathcal{O}))$ consisting of $\mathfrak{a}$-torsion elements. 
See \cite{kuroda2024characteristic-ecaa} for the details of the properties above. 

The first and third authors generalized \zcref{LTY} and \zcref{HTY} as follows. 

\begin{theorem}[{\cite[Theorem 3.2, Theorem 4.3, and Theorem 5.1]{kuroda2024characteristic-ecaa}}]
The $\kappa$-constituent $\chi_{\mathcal{A}}^{\kappa}(t)$ of $\chi_{\mathcal{A}}^{\mathrm{quasi}}(\mathfrak{a})$ coincides with the characteristic polynomial of $L(\mathcal{A} (K / \mathcal{O}) )[\kappa]$. 
\end{theorem}

\begin{remark}
The characteristic quasi-polynomials of Weyl arrangements associated with crystallographic root systems have specifically interesting properties (See \cite{kamiya2010characteristic-alsas, yoshinaga2018worpitzky-tmj}). 
By virtue of generalizing the characteristic quasi-polynomials for arrangement over residually finite Dedekind domains, we can consider the characteristic quasi-polynomials of arrangements associated non-crystallographic root systems and complex reflection groups. 
The characteristic quasi-polynomials for exceptional well-generated complex reflection groups are studied in \cite{kuroda2024characteristic-sldc}. 
\end{remark}

\subsection{Main results}
Let $\mathcal{O}$ be a residually finite Dedekind domain and $K$ its field of fractions. 
Let $\mathcal{A} = \{\alpha_{1}, \dots, \alpha_{n}\}$ be a finite set of nonzero vectors in $\mathcal{O}^{\ell}$. 
Two vectors $\alpha, \beta \in \mathcal{A}$ are called \textbf{parallel} if $\alpha = k\beta$ for some $k \in K^{\times}$. 
The parallelism is an equivalence relation on $\mathcal{A}$ and we call an equivalence class a \textbf{parallel class}. 
Let $\mathcal{P}(\mathcal{A})$ be the set of parallel classes of $\mathcal{A}$ and put $p(\mathcal{A}) \coloneqq \#\mathcal{P}(\mathcal{A})$.

Let $\mathfrak{a} \in I(\mathcal{O})$. 
Given a function $f \in \left(\mathcal{O}/\mathfrak{a}\right)^{\mathcal{A}}$, we define its \textbf{weight} $\wt(f)$ to be the number of parallel classes $C$ such that $f(\alpha) \neq 0$ for any $\alpha \in C$. 
We will naturally identify $\left(\mathcal{O}/\mathfrak{a}\right)^{\mathcal{A}}$ with $\left(\mathcal{O}/\mathfrak{a}\right)^{n}$ and consider the induced weight. 
If each parallel class is a singleton, then the weight coincides with the Hamming weight on $\left(\mathcal{O}/\mathfrak{a}\right)^{n}$.

Define the matrix $G$ by $G \coloneqq (\alpha_{1} \ \cdots \ \alpha_{n}) \in \Mat_{\ell \times n}\left(\mathcal{O}\right)$. 
For any $\mathfrak{a} \in I(\mathcal{O})$, the matrix $G$ induces the homomorphism $G_{\mathfrak{a}} \colon \left(\mathcal{O}/\mathfrak{a}\right)^{\ell} \to \left(\mathcal{O}/\mathfrak{a}\right)^{n}$ by $[u]_{\mathfrak{a}} \mapsto [u]_{\mathfrak{a}}G$, where $[u]_{\mathfrak{a}}$ denotes the residue class of a row vector $u \in \mathcal{O}^{\ell}$.

\begin{definition}
Define the \textbf{weight enumerator} of $\Image G_{\mathfrak{a}}$ by 
\begin{align*}
W_{\Image G_{\mathfrak{a}}}(x,y) \coloneqq \sum_{c \in \Image G_{\mathfrak{a}}}x^{p(\mathcal{A})-\wt(c)}y^{\wt(c)}.  
\end{align*}
\end{definition}

\begin{definition}
Define the map $\overline{\chi}_{\mathcal{A}}^{\mathrm{quasi}} \colon I(\mathcal{O}) \to \mathbb{Z}[x]$ by 
\begin{align*}
\overline{\chi}_{\mathcal{A}}^{\mathrm{quasi}}(\mathfrak{a}, x) \coloneqq \sum_{i=0}^{p(\mathcal{A})}B_{i}(\mathfrak{a})x^{i}, 
\end{align*}
where 
\begin{align*}
B_{i}(\mathfrak{a}) \coloneqq \# \Set{[u]_{\mathfrak{a}} \in \left(\mathcal{O}/\mathfrak{a}\right)^{\ell} | \wt([u]_{\mathfrak{a}} G) = p(\mathcal{A})-i}. 
\end{align*}
\end{definition}
The $\overline{\chi}_{\mathcal{A}}^{\mathrm{quasi}}$ is a quasi-polynomial on $I(\mathcal{O})$ and every constituent of it is the coboundary polynomial of the torsion subposet of the poset of layers (see \zcref{main theorem}). 
Thus we call it \textbf{coboundary quasi-polynomial} of $\mathcal{A}$.  
The following proposition is an analogy of \zcref{Greene}.

\begin{proposition}\label{coboundary quasi-polynomial and weight enumerator}
For any $\mathfrak{a} \in I(\mathcal{O})$, 
\begin{align*}
\overline{\chi}_{\mathcal{A}}^{\mathrm{quasi}}(\mathfrak{a}, x)
= \# \Kernel G_{\mathfrak{a}} \cdot W_{\Image G_{\mathfrak{a}}}(x,1). 
\end{align*}
\end{proposition}
\begin{proof}
The isomorphism $\left(\mathcal{O}/\mathfrak{a}\right)^{\ell}/\Kernel G_{\mathfrak{a}} \simeq \Image G_{\mathfrak{a}}$ implies 
\begin{align*}
B_{i}(\mathfrak{a}) = \# \Kernel G_{\mathfrak{a}} \cdot \# \Set{c \in \Image G_{\mathfrak{a}} | p(\mathcal{A}) - \wt(c) = i}. 
\end{align*}
Therefore 
\begin{align*}
\# \Kernel G_{\mathfrak{a}} \cdot W_{\Image G_{\mathfrak{a}}}(x,1) 
&= \# \Kernel G_{\mathfrak{a}} \sum_{c \in \Image G_{\mathfrak{a}}} x^{p(\mathcal{A})-\wt(c)} \\
&= \# \Kernel G_{\mathfrak{a}} \sum_{i=0}^{p(\mathcal{A})} \#\Set{c \in \Image G_{\mathfrak{a}} | p(\mathcal{A}) - \wt(c) = i}x^{i} \\
&= \sum_{i=0}^{p(\mathcal{A})}B_{i}(\mathfrak{a})x^{i} 
= \overline{\chi}_{\mathcal{A}}^{\mathrm{quasi}}(\mathfrak{a}, x). 
\end{align*}
\end{proof}

The main results of this article are as follows. 
\begin{theorem}\label{main theorem}
Let $\rho_{\mathcal{A}}$ be the LCM-period. 
The following statements hold. 
\begin{enumerate}[(1)]
\item\label{main theorem 1} $\overline{\chi}_{\mathcal{A}}^{\mathrm{quasi}}(\mathfrak{a}, 0) =  \chi_{\mathcal{A}}^{\mathrm{quasi}}(\mathfrak{a})$ for any $\mathfrak{a} \in I(\mathcal{O})$. 
\item\label{main theorem 2} Every $B_{i}(\mathfrak{a}) $ is a quasi-polynomial in $\mathfrak{a}$ of degree at most $\ell$ with period $\rho_{\mathcal{A}}$. 
\item\label{main theorem 3} $\overline{\chi}_{\mathcal{A}}^{\mathrm{quasi}}(\mathfrak{a},x)$ is a quasi-polynomial in $\mathfrak{a}$ with minimum period $\rho_{\mathcal{A}}$. 
\item\label{main theorem 4} The $\kappa$-constituent $\overline{\chi}_{\mathcal{A}}^{\kappa}(t,x)$ of $\overline{\chi}_{\mathcal{A}}^{\mathrm{quasi}}(\mathfrak{a},x)$ coincides with the coboundary polynomial of $L(\mathcal{A} (K / \mathcal{O}) )[\kappa]$. 
Namely 
\begin{align*}
\overline{\chi}_{\mathcal{A}}^{\kappa}(t,x) = \sum_{Z \in L(\mathcal{A}(K/\mathcal{O}))[\kappa]}\left(\sum_{\substack{Y \in L(\mathcal{A}(K/\mathcal{O}))[\kappa] \\ Y \geq Z}}\mu(Z,Y)t^{\dim Y}\right)x^{a(Z)}, 
\end{align*}
where $a(Z)$ denotes the number of atoms (minimal elements in $L(\mathcal{A}(K/\mathcal{O})) \setminus \{\hat{0}\}$) below $Z$. 
\item\label{main theorem 5} The $\mathcal{O}$-constituent of $\overline{\chi}_{\mathcal{A}}^{\mathrm{quasi}}(\mathfrak{a},x)$ coincides with $\overline{\chi}_{\mathcal{A}(K)}(t,x)$. 
\end{enumerate}
\end{theorem}

\begin{remark}
When $\mathcal{O}=\mathbb{Z}$, we can define the coboundary quasi-polynomial in positive integers $q$ by setting $\overline{\chi}_{\mathcal{A}}^{\mathrm{quasi}}(q,x) \coloneqq \overline{\chi}_{\mathcal{A}}^{\mathrm{quasi}}(q\mathbb{Z},x)$. 
Then $\overline{\chi}_{\mathcal{A}}^{\mathrm{quasi}}(q,x)$ is a quasi-polynomial in $q$ in the usual sense. 
Note that the equivalent notion that is called the \textbf{Tutte quasi-polynomial} was introduced in \cite{branden2014multivariate-totams}. 
\end{remark}

\begin{example}[Coboundary quasi-polynomial of $\mathrm{H}_{3}$]
Let $\tau = \frac{1+\sqrt{5}}{2}$ be the golden ratio and let $\mathcal{O} = \mathbb{Z}[\tau]$. 
In the previous paper \cite[Section 8.2]{kuroda2024characteristic-ecaa}, the characteristic quasi-polynomial $\chi_{\mathrm{H}_{3}}^{\mathrm{quasi}}$ of the root system of type $\mathrm{H}_{3}$ was computed, where by the root system of type $\mathrm{H}_{3}$ we mean the set
\begin{align*}
\Phi_{\mathrm{H}_{3}} = \Set{\begin{array}{cl}
(\pm 1, 0, 0) & \text{ and all permutations} \\
\frac{1}{2}(\pm \tau, \pm 1, \pm \tau^{-1}) & \text{ and all even permutations}
\end{array} }. 
\end{align*}
We can choose the following simple roots:  
\begin{align*}
\alpha_{1} = \frac{1}{2}(\tau, -1, \tau^{-1}), \qquad
\alpha_{2} = \frac{1}{2}(-\tau, 1, \tau^{-1}), \qquad
\alpha_{3} = \frac{1}{2}(1, \tau^{-1}, -\tau). 
\end{align*}
Then the coefficient matrix of the positive roots with respect to $\{\alpha_{1}, \alpha_{2}, \alpha_{3}\}$ is 
\begin{align*}
G = \begin{pmatrix}
1 & 0 & \tau & 1 & \tau & 0 & 0 & \tau & \tau & \tau^{2} & 1 & \tau & \tau & \tau^{2} & \tau^{2} \\
0 & 1 & 1 & \tau & \tau & 0 & 1 & 1 & \tau^{2} & \tau^{2} & \tau & \tau & \tau^{2} & \tau^{2} & 2 \tau \\
0 & 0 & 0 & 0 & 0 & 1 & 1 & 1 & 1 & 1 & \tau & \tau & \tau & \tau & \tau
\end{pmatrix}. 
\end{align*}
Let $\mathcal{A}_{\mathrm{H}_{3}}$ be the set consisting of the column vectors of the matrix above. 
Then the LCM-period is $\rho_{\mathrm{H}_{3}} = \langle 2 \rangle$ and the constituents of $\chi_{\mathrm{H}_{3}}^{\mathrm{quasi}}$ are given by  
\begin{align*}
\chi^{\langle 1 \rangle}_{\mathrm{H}_3} (t) &= t^{3}-15t^{2}+59t-45 = (t-1)(t-5)(t-9), \\
\chi^{\langle 2 \rangle}_{\mathrm{H}_3} (t) &= t^{3}-15t^{2}+59t-60 = (t-4)(t^{2}-11t+15). 
\end{align*}

Now, we compute the coboundary quasi-polynomial $\overline{\chi}_{\mathrm{H}_{3}}^{\mathrm{quasi}}$. 
Note that $p(\mathcal{A}_{\mathrm{H}_{3}}) = \# \mathcal{A}_{\mathrm{H}_{3}} = 15$ since there is no pair of parallel vectors in $\mathcal{A}_{\mathrm{H}_{3}}$. 
Since every $B_{i}(\mathfrak{a})$ is a quasi-polynomial in $\mathfrak{a}$ of degree at most $3$ by \zcref{main theorem} (\ref{main theorem 2}), we can determine the constituents of $B_{i}(\mathfrak{a})$ by interpolation using sufficiently many values.
The following table lists the values of $B_{i}(\langle q \rangle)$ where $q \in \{1,2, \dots, 10\}$.  
\[
\begin{array}{c|rrrrrrrrrrrrrrrr}
q \backslash i & 0 & 1 & 2 & 3 & 4 & 5 & 6 & 7 & 8 & 9 & 10 & 11 & 12 & 13 & 14 & 15 \\
\hline 
1 & 0 & 0 & 0 & 0 & 0 & 0 & 0 & 0 & 0 & 0 & 0 & 0 & 0 & 0 & 0 & 1\\
2 & 0 & 0 & 0 & 45 & 0 & 18 & 0 & 0 & 0 & 0 & 0 & 0 & 0 & 0 & 0 & 1\\
3 & 0 & 480 & 120 & 80 & 0 & 48 & 0 & 0 & 0 & 0 & 0 & 0 & 0 & 0 & 0 & 1\\
4 & 1140 & 2520 & 180 & 165 & 0 & 90 & 0 & 0 & 0 & 0 & 0 & 0 & 0 & 0 & 0 & 1\\
5 & 7680 & 7200 & 360 & 240 & 0 & 144 & 0 & 0 & 0 & 0 & 0 & 0 & 0 & 0 & 0 & 1\\
6 & 29280 & 16320 & 480 & 365 & 0 & 210 & 0 & 0 & 0 & 0 & 0 & 0 & 0 & 0 & 0 & 1\\
7 & 84480 & 31680 & 720 & 480 & 0 & 288 & 0 & 0 & 0 & 0 & 0 & 0 & 0 & 0 & 0 & 1\\
8 & 204420 & 55800 & 900 & 645 & 0 & 378 & 0 & 0 & 0 & 0 & 0 & 0 & 0 & 0 & 0 & 1\\
9 & 437760 & 91200 & 1200 & 800 & 0 & 480 & 0 & 0 & 0 & 0 & 0 & 0 & 0 & 0 & 0 & 1\\
10 & 855840 & 141120 & 1440 & 1005 & 0 & 594 & 0 & 0 & 0 & 0 & 0 & 0 & 0 & 0 & 0 & 1
\end{array}
\]

The constituents of $\overline{\chi}_{\mathrm{H}_{3}}^{\mathrm{quasi
}}$ are given by 
\begin{align*}
\overline{\chi}_{\mathrm{H}_{3}}^{\langle 1 \rangle}(t,x) 
  &= (t-1)(t-5)(t-9) + 15(t-1)(t-5)x + 15(t-1)x^{2} \\
  &\quad + 10(t-1)x^{3} + 6(t-1)x^{5} + x^{15}, \\
\overline{\chi}_{\mathrm{H}_{3}}^{\langle 2 \rangle}(t,x) 
  &= (t-4)(t^{2}-11t+15) + 15(t-2)(t-4)x + 15(t-4)x^{2} \\
  &\quad + 5(2t+1)x^{3} + 6(t-1)x^{5} + x^{15}.
\end{align*}
Note that $N(\langle q \rangle) = q^{2}$. 

For any ideal $\mathfrak{a} \in I(\mathcal{O})$, $\Image G_{\mathfrak{a}}$ is a linear code in $(\mathcal{O}/\mathfrak{a})^{15}$. 
Since $\# \Kernel G_{\mathfrak{a}} = 1$, the weight enumerator of $\Image G_{\mathfrak{a}}$ is given by $W_{\Image G_{\mathfrak{a}}}(x,1) = \overline{\chi}_{\mathrm{H}_{3}}^{\mathrm{quasi}}(\mathfrak{a}, x)$ by \zcref{coboundary quasi-polynomial and weight enumerator}. 
\end{example}

\begin{example}[A non-PID case]
Let $\mathcal{O} = \mathbb{Z}[\sqrt{-5}]$ and $K = \mathbb{Q}(\sqrt{-5})$. 
Then $\mathcal{O}$ is a residually finite Dedekind domain that is not a principal ideal domain. 
Let $\mathfrak{p}$ and $\mathfrak{q}$ be prime ideals of $\mathcal{O}$ defined by $\mathfrak{p} = \langle 2, 1-\sqrt{-5}\rangle, \mathfrak{q} = \langle 3, 1+\sqrt{-5}\rangle$. 
Let us consider the case 
\begin{align*}
\mathcal{A} = \{\alpha_{1}, \alpha_{2}\}, \text{ where } 
\alpha_{1} = \begin{pmatrix}
2 \\ 1-\sqrt{-5}
\end{pmatrix}, \
\alpha_{2} = \begin{pmatrix}
1+\sqrt{-5} \\ 3
\end{pmatrix}. 
\end{align*}
In the previous paper \cite[Example 3.2]{kuroda2024characteristic-ecaa}, the LCM-period and the characteristic quasi-polynomial of $\mathcal{A}$ are determined. 
The LCM-period $\rho_{\mathcal{A}}$ of $\mathcal{A}$ equals $\mathfrak{pq} = \langle 1+\sqrt{-5} \rangle$ and the constituents are given by 
\begin{align*}
\chi_{\mathcal{A}}^{\mathcal{O}}(t) = t^{2}-t, \qquad
\chi_{\mathcal{A}}^{\mathfrak{p}}(t) = t^{2}-2t, \qquad
\chi_{\mathcal{A}}^{\mathfrak{q}}(t) = t^{2}-3t, \qquad
\chi_{\mathcal{A}}^{\mathfrak{pq}}(t) = t^{2}-4t. 
\end{align*}

Note that $\alpha_{1}$ and $\alpha_{2}$ are parallel since $3\alpha_{1} = (1-\sqrt{-5})\alpha_{2}$. 
Therefore $p(\mathcal{A}) = 1$. 
Since $\overline{\chi}_{\mathcal{A}}^{\mathrm{quasi}}(\mathfrak{a},1) = N(\mathfrak{a})^{2}$ for any $\mathfrak{a} \in I(\mathcal{O})$, the constituents of the coboundary quasi-polynomial of $\mathcal{A}$ are given by 
\begin{align*}
\overline{\chi}_{\mathcal{A}}^{\mathcal{O}}(t,x) &= (t^{2}-t) + tx, 
&
\overline{\chi}_{\mathcal{A}}^{\mathfrak{p}}(t,x) = (t^{2}-2t) + 2tx, \\
\overline{\chi}_{\mathcal{A}}^{\mathfrak{q}}(t,x) &= (t^{2}-3t) + 3tx, 
&
\overline{\chi}_{\mathcal{A}}^{\mathfrak{pq}}(t,x) = (t^{2}-4t) + 4tx. 
\end{align*}

From now on, we determine the poset of layers to confirm the constituents above coincide with the coboundary polynomial of the corresponding posets (\zcref{main theorem}(\ref{main theorem 4})). 
Since $\alpha_{1}$ and $\alpha_{2}$ are parallel, 
\begin{align*}
H \coloneqq H_{\alpha_{1}, K} = H_{\alpha_{2}, K} = H_{\{\alpha_{1}, \alpha_{2}\}, K} = \langle (-3,1+\sqrt{-5}) \rangle_{K}. 
\end{align*} 

We show that 
\begin{align*}
\pi(H) &= \Set{ \pi(a,b) \in \left(K/\mathcal{O}\right)^{2} | 2a+(1-\sqrt{-5})b \in \mathfrak{p}} \\
&= \Set{ \pi(a,b) \in \left(K/\mathcal{O}\right)^{2} | (1+\sqrt{-5})a+3b \in \mathfrak{q}}. 
\end{align*}
We prove only the first equality since the proof of the second equality is similar. 
Suppose that $\pi(a,b) \in \pi(H)$, where $(a,b) \in H$. 
Then $2a + (1-\sqrt{-5})b = 0 \in \mathfrak{p}$. 
Hence $\pi(a,b)$ belongs to the right hand side. 
Conversely, suppose that $\pi(a,b)$ belongs to the right hand side. 
Then there exists $(\alpha, \beta) \in \mathcal{O}^{2}$ such that $2a + (1-\sqrt{-5})b = 2\alpha + (1-\sqrt{-5})\beta$. 
Then $(a-\alpha, b-\beta) \in H$, which implies that $\pi(a,b) = \pi(a-\alpha, b-\beta) \in \pi(H)$. 
Thus the fist equality holds. 

We next see that  
\begin{align*}
H_{\varnothing, K/\mathcal{O}}/\pi(H_{\varnothing,K}) &= \Set{\left(K/\mathcal{O}\right)^{2}}, \\
H_{\alpha_{1}, K/\mathcal{O}}/\pi(H_{\alpha_{1},K}) &= \Set{\overline{(0,0)}, \ \overline{(1/2, 0)}}, \\
H_{\alpha_{2}, K/\mathcal{O}}/\pi(H_{\alpha_{2},K}) &= \Set{\overline{(0,0)}, \ \overline{(0, 1/3)}, \ \overline{(0, 2/3)}},  \\
H_{\{\alpha_{1}, \alpha_{2}\}, K/\mathcal{O}}/\pi(H_{\{\alpha_{1}, \alpha_{2}\},K}) &= \Set{\overline{(0,0)}}, 
\end{align*}
where $\overline{(a,b)}$ denotes  $\pi(a,b) + \pi(H)$. 
We prove only the second equality since the others are similar. 
Suppose that $\overline{(a,b)}$ belongs to the left hand side. 
Let $\alpha = (1-\sqrt{-5})b/6 \in K$. 
Then $(-(1-\sqrt{-5})b/2, b) = (-3\alpha, (1+\sqrt{-5})\alpha) \in H$. 
Hence 
\begin{align*}
\overline{(a,b)} = \overline{(a,b)} - \overline{\left(-\frac{1-\sqrt{-5}}{2}b, b\right)}
= \overline{\left(a+\frac{1-\sqrt{-5}}{2}b, 0\right)}. 
\end{align*}
Hence without loss of generality we can assume that $b=0$. 
Since $\pi(a,0) \in H_{\alpha_{1}, K/\mathcal{O}}$, $2a \in \mathcal{O}$. 
Therefore there exist $\alpha, \beta \in \mathbb{Z}$ such that $2a = \alpha + \beta\sqrt{-5}$. 
If $\alpha+\beta$ is even, then $2a = \alpha + \beta -(1-\sqrt{-5})\beta \in \mathfrak{p}$. 
Hence $\pi(a,0) \in \pi(H)$, which implies that $\overline{(a,0)} = \overline{(0,0)}$. 
If $\alpha+\beta$ is odd, then $2a-1 = \alpha + \beta - 1 -(1-\sqrt{-5})\beta \in \mathfrak{p}$. 
Hence $\pi(a-1/2,0) \in \pi(H)$, which implies that $\overline{(a,0)} = \overline{(1/2,0)}$. 
Thus the second equality holds. 

The Hasse diagram of the poset of layers $L(\mathcal{A}(K/\mathcal{O}))$ is as follows. 
\begin{center}
\begin{tikzpicture}
  \node (0) at (-3,0) {$\overline{(0,0)}$};
  \node (11) at (-1,0) {$\overline{(1/2,0)}$};
  \node (21) at (1,0) {$\overline{(0,1/3)}$};
  \node (22) at (3,0) {$\overline{(0,2/3)}$};
  \node (empty) at (0,-2) {$\left(K/\mathcal{O}\right)^{2}$};
  \draw (empty) -- (0);
  \draw (empty) -- (11);
  \draw (empty) -- (21);
  \draw (empty) -- (22);
\end{tikzpicture}
\end{center}

For each $\kappa \mid \rho_{\mathcal{A}}$, the $\kappa$-torsion subposet is as follows. 
\begin{align*}
L(\mathcal{A}(K/\mathcal{O}))[\mathcal{O}] &=  \Set{\left(K/\mathcal{O}\right)^{2}, \overline{(0,0)}}, \\
L(\mathcal{A}(K/\mathcal{O}))[\mathfrak{p}] &= \Set{\left(K/\mathcal{O}\right)^{2}, \overline{(0,0)}, \overline{(1/2,0)}}, \\
L(\mathcal{A}(K/\mathcal{O}))[\mathfrak{q}] &= \Set{\left(K/\mathcal{O}\right)^{2}, \overline{(0,0)}, \overline{(0,1/3)}, \overline{(0,2/3)}}, \\
L(\mathcal{A}(K/\mathcal{O}))[\mathfrak{pq}] &= L(\mathcal{A}(K/\mathcal{O})). 
\end{align*}
We can verify that the coboundary polynomials of the poset above coincide with the corresponding constituents of the coboundary quasi-polynomial. 
\end{example}

The organization of this paper is as follows. 
In \zcref{sec: preliminaries}, we prove six lemmas for the main results after a brief review of layers for arrangements over residually finite Dedekind domains. 
In \zcref{sec:proof}, we give a proof of \zcref{main theorem}. 
In \zcref{sec:question}, we raise a question about the MacWilliams identities and the dual matroid over a ring introduced by Fink and Moci \cite{fink2016matroids-jotems}.

\section{Preliminaries}\label{sec: preliminaries}
To ease notation we write $L = L(\mathcal{A}(K/\mathcal{O}))$ and $L[\mathfrak{a}] = L(\mathcal{A}(K/\mathcal{O}))[\mathfrak{a}]$ in this section. 

\subsection{Brief review of layers for arrangements over residually finite Dedekind domains}

We give a brief review of the known properties of layers that are used in this article.
Note that \zcref{KT Prop 4.2} and \zcref{KT Lem 4.3} are intuitively clear for toric arrangements because of its topological property. 

\begin{proposition}[{\cite[Proposition 4.2]{kuroda2024characteristic-ecaa}}]\label{KT Prop 4.2}
Let $Z \in L$ be a layer and suppose that $Z \subseteq H_{J,K/\mathcal{O}}$ for some $J \subseteq \mathcal{A}$. 
Then there exists a layer $W \in H_{J,K/\mathcal{O}}/\pi(H_{J,K})$ such that $Z \subseteq W$. 
\end{proposition}

\begin{proposition}[{\cite[Lemma 4.3 and Lemma 4.6]{kuroda2024characteristic-ecaa}}]\label{KT Lem 4.3}
Given a layer $Z \in L$, there exists a unique intersection $H_{Z} \in L(\mathcal{A}(K))$ such that $Z = \pi(u) + \pi(H_{Z})$ for any $\pi(u) \in Z$. 
In particular $H_{Z} = H_{J_{Z},K}$, where $J_{Z} \coloneqq \Set{\alpha \in \mathcal{A} | Z \subseteq H_{\alpha, K/\mathcal{O}}}$. 
\end{proposition}

The torsion subposets of the poset of layers are periodic with respect to the LCM-period (\zcref{KT Cor4.1}) and the $\mathcal{O}$-torsion subposet is isomorphic to the intersection lattice of the corresponding hyperplane arrangement (\zcref{KT Cor4.2}). 

\begin{proposition}[{\cite[Corollary 4.1]{kuroda2024characteristic-ecaa}}]\label{KT Cor4.1}
For any $\mathfrak{a} \in I(\mathcal{O})$, $L[\mathfrak{a} + \rho_{\mathcal{A}}] = L[\mathfrak{a}]$. 
\end{proposition}

\begin{proposition}[{\cite[Corollary 4.2]{kuroda2024characteristic-ecaa}}]\label{KT Cor4.2}
$L[\mathcal{O}] \simeq L(\mathcal{A}(K))$ and $L[\rho_{\mathcal{A}}] = L$. 
\end{proposition}

\subsection{Six lemmas}

In this subsection, we prove six lemmas which will be used in the proof of \zcref{main theorem}.
Recall that a layer $Z \in L$ is an $\mathfrak{a}$-torsion element if $Z$ is an $\mathfrak{a}$-torsion element in the module $H_{J,K/\mathcal{O}}/\pi(H_{J,K})$ when $Z \in H_{J,K/\mathcal{O}}/\pi(H_{J,K})$. 
Note that being an $\mathfrak{a}$-torsion element is independent of the choice of $J$. 

The first lemma that we prove in this section is known to hold when $\mathcal{O}=\mathbb{Z}$ (see \cite[Definition 4.4(1) and (4.3) in the proof of Lemma 4.5]{tran2019combinatorics-joctsa}).
Moreover, its validity for general residually finite Dedekind domains was raised as a problem in \cite[Remark 4.2]{kuroda2024characteristic-ecaa}. 

\begin{lemma}\label{torsion lemma}
Let $Z \in L$ and $\mathfrak{a} \in I(\mathcal{O})$. 
Then the following are equivalent.  
\begin{enumerate}[(1)]
\item\label{torsion lemma 1} $Z$ is an $\mathfrak{a}$-torsion element. 
\item\label{torsion lemma 2} $Z$ has an $\mathfrak{a}$-torsion element in $\left(K/\mathcal{O}\right)^{\ell}$. 
\end{enumerate}
\end{lemma}
\begin{proof}
Suppose that $Z$ belongs to the module $H_{J,K/\mathcal{O}}/\pi(H_{J,K})$, where $J = \{\alpha_{j_{1}}, \dots, \alpha_{j_{k}}\} \subseteq \mathcal{A}$. 
Define the submatrix $G_{J}$ of $G$ by $G_{J} \coloneqq (\alpha_{j_{1}} \ \dots \alpha_{j_{k}} )$. 
We identify the matrix $G_{J}$ with the map from $\mathcal{O}^{\ell}$ to $\mathcal{O}^{k}$ defined by $u \mapsto uG_{J}$. 

First, we show that (\ref{torsion lemma 1}) implies (\ref{torsion lemma 2}). 
Let $\pi(u) \in H_{J,K/\mathcal{O}}$ be a representative of $Z$. 
Put $v \coloneqq uG_{J} \in \mathcal{O}^{k}$. 
We show that $\mathfrak{a}v \subseteq \Image G_{J}$.  
Take an element $a \in \mathfrak{a}$. 
Since $Z$ is an $\mathfrak{a}$-torsion element, $a\pi(u) \in \pi(H_{J,K})$. 
Therefore there exists $w \in H_{J,K}$ such that $au-w \in \mathcal{O}^{\ell}$. 
Then $(au-w)G_{J} = auG_{J} - wG_{J} = av - 0 = av$. 
Hence $av \in \Image G_{J}$. 
Therefore $\mathfrak{a}v \subseteq \Image G_{J}$. 

Since $v \in \mathcal{O}v = \mathfrak{a}^{-1}\mathfrak{a}v \subseteq \mathfrak{a}^{-1} \Image G_{J}$, $v$ can be represented as 
\begin{align*}
v = \sum_{i=1}^{m}b_{i} \cdot w_{i}G_{J} = \left(\sum_{i=1}^{m}b_{i}w_{i}\right)G_{J}, 
\end{align*}
where $b_{i} \in \mathfrak{a}^{-1}$ and $w_{i} \in \mathcal{O}^{\ell}$ for each $i$. 
Put $z \coloneqq \sum_{i=1}^{m}b_{i}w_{i} \in K^{\ell}$. 
Then $zG_{J} = v = uG_{J}$. 
Therefore $z-u \in H_{J,K}$ and $\pi(z)-\pi(u) \in \pi(H_{J,K})$. 
Hence $\pi(z) \in \pi(u) + \pi(H_{J,K}) = Z$. 
Moreover, $az = \sum_{i=1}^{m}ab_{i}w_{i} \in \mathcal{O}^{\ell}$ for any $a \in \mathfrak{a}$. 
Therefore $Z$ has an $\mathfrak{a}$-torsion element $\pi(z)$. 

Next, assume (\ref{torsion lemma 2}) holds. 
Suppose that $\pi(u) \in Z$ is an $\mathfrak{a}$-torsion element. 
Then for any $a \in \mathfrak{a}$, $a \cdot Z = \pi(au) + \pi(H_{J,K}) = \pi(H_{J,K})$. 
Thus $Z$ is an $\mathfrak{a}$-torsion element in the module $H_{J,K/\mathcal{O}}/\pi(H_{J,K})$. 
\end{proof}

Our next goal is to show that the cardinality of the $\mathfrak{a}$-torsion part of a layer $Z$ is equal to the absolute norm $N(\mathfrak{a})$ whenever $Z[\mathfrak{a}] \neq \varnothing$ (see \zcref{cardinality of Z[a]}).  
To this end, we first prove the statement in the case where $\mathcal{O}$ is a principal ideal domain (see \zcref{pid case}).

\begin{lemma}\label{pid case}
Suppose that $\mathcal{O}$ is a principal ideal domain. 
Let $U$ be a subspace of $K^{\ell}$. 
Then 
$\# \pi(U)[a] = N(\langle a \rangle)^{\dim_{K}U}$ for any $a \in \mathcal{O}\setminus\{0\}$. 
\end{lemma}
\begin{proof}
Let $r \coloneqq \dim_{K}U$. 
Since $U$ is the kernel of the canonical projection $K^{\ell} \to K^{\ell}/U$, there exists a matrix $A \in \Mat_{\ell \times (\ell-r)}(\mathcal{O})$ with $\rank A = \ell-r$ such that
\begin{align*}
U = \Set{u \in K^{\ell} | uA = 0}. 
\end{align*}
Let $P \in \GL_{\ell}(\mathcal{O})$ and $Q \in \GL_{\ell-r}(\mathcal{O})$ be matrices such that $PAQ$ is the Smith normal form of $A$. 
Define 
\begin{align*}
U^{\prime} \coloneqq \Set{u \in K^{\ell} | uPAQ = 0} = \Set{u \in K^{\ell} | u_{1} = \dots = u_{\ell-r} = 0}. 
\end{align*}
The spaces $U$ and $U^{\prime}$ are isomorphic as $\mathcal{O}$-modules by the map $u \to uP^{-1}$, which induces an isomorphism $\pi(U) \simeq \pi(U^{\prime})$ as $\mathcal{O}$-modules. 
Therefore 
\begin{align*}
\#\pi(U)[a] = \#\pi(U^{\prime})[a] = \# \left(a^{-1}\mathcal{O}/\mathcal{O}\right)^{r} = \#\left(\mathcal{O}/\langle a \rangle\right)^{r} = N(\langle a \rangle)^{\dim_{K}U}. 
\end{align*}
\end{proof}

\begin{lemma}\label{cardinality of Z[a]}
Let $\mathfrak{a} \in I(\mathcal{O})$ and $Z \in L$. 
Suppose that $Z[\mathfrak{a}] \neq \varnothing$.
Then $\# Z[\mathfrak{a}] = N(\mathfrak{a})^{\dim Z}$. 
\end{lemma}
\begin{proof}
Suppose that $Z$ belongs to the module $H_{J,K/\mathcal{O}}/\pi(H_{J,K})$. 
Let $\pi(u) \in Z[\mathfrak{a}]$. 
Then $Z = \pi(u) + \pi(H_{J,K})$ and $u \in (\mathfrak{a}^{-1})^{\ell}$. 
Define the map $\varphi \colon \pi(H_{J,K})[\mathfrak{a}] \to Z[\mathfrak{a}]$ by 
$\varphi(\pi(w)) \coloneqq \pi(u) + \pi(w)$. 
Since $\phi$ is bijection, $\# Z[\mathfrak{a}] = \# \pi(H_{J,K})[\mathfrak{a}]$.

Let $S \coloneqq \mathcal{O}\setminus\bigcup_{\mathfrak{p}} \mathfrak{p}$, where $\mathfrak{p}$ runs over all prime ideals in $\mathcal{O}$ with $\mathfrak{p} \mid \mathfrak{a}$. 
Then the localization $\mathcal{O}_{S}$ of $\mathcal{O}$ with respect to $S$ is a principal ideal domain.
Let $\pi_{S} \colon K^{\ell} \to \left(K/\mathcal{O}_{S}\right)^{\ell}$ be the canonical projection 
and $\eta \colon (K/\mathcal{O})^{\ell} \to (K/\mathcal{O}_{S})^{\ell}$ be the $\mathcal{O}$-homomorphism defined by $\eta(\pi(u)) \coloneqq \pi_{S}(u)$. 
Suppose that $\pi(u) \in \pi(H_{J,K})[\mathfrak{a}]$ with $u \in H_{J,K}$. 
Then, for any $a/s \in \mathfrak{a}\mathcal{O}_{S}$ with $a \in \mathfrak{a}$ and $s \in S$, we obtain that $au \in \mathcal{O}^{\ell}$, which implies that $\frac{1}{s}au \in \mathcal{O}_{S}^{\ell}$. 
Thus 
\begin{align*}
\dfrac{a}{s} \eta(\pi(u)) = \dfrac{a}{s}\pi_{S}(u) = \pi_{S}\left(\dfrac{1}{s}au\right) = \pi_{S}(0). 
\end{align*}
Therefore $\eta(\pi(u)) \in \pi_{S}(H_{J,K})[\mathfrak{a}\mathcal{O}_{S}]$ and $\eta$ induces an $\mathcal{O}$-homomorphism from $\pi(H_{J,K})[\mathfrak{a}]$ to $\pi_{S}(H_{J,K})[\mathfrak{a}\mathcal{O}_{S}]$. 
We will show that this map, also denoted by $\eta$, is an $\mathcal{O}$-isomorphism. 

To show the injectivity, suppose that $\pi(u) \in \pi(H_{J,K})[\mathfrak{a}]$ satisfies $\eta(\pi(u)) = \pi_{S}(0)$. 
Then $u \in \mathcal{O}_{S}^{\ell}$ and there exists $s \in S$ such that $su \in \mathcal{O}^{\ell}$. 
Since $\langle s \rangle + \mathfrak{a} = \mathcal{O}$ and $\mathfrak{a}u \subseteq \mathcal{O}^{\ell}$ 
\begin{align*}
u \in \mathcal{O}u = (\langle s \rangle + \mathfrak{a})u = \langle s \rangle u + \mathfrak{a}u \subseteq \mathcal{O}^{\ell}. 
\end{align*}
Therefore $\pi(u) = \pi(0)$, which implies that $\eta$ is injective. 

To prove that $\eta$ is surjective, suppose that $\pi_{S}(u) \in \pi_{S}(H_{J,K})[\mathfrak{a}\mathcal{O}_{S}]$ with $u \in H_{J,K}$. 
Then $\mathfrak{a}u \subseteq \mathcal{O}_{S}^{\ell}$. 
Therefore $u \in \mathcal{O}u = \mathfrak{a}^{-1}\mathfrak{a}u \subseteq \mathfrak{a}^{-1}\mathcal{O}_{S}^{\ell}$. 
Hence there exists $s \in S$ such that $su \in (\mathfrak{a}^{-1})^{\ell}$. 
Since $\mathcal{O} = \langle s \rangle + \mathfrak{a}$, there exist $b \in \mathcal{O}$ and $a \in \mathfrak{a}$ such that $1 = bs + a$. 
Then $u = bsu + au$. 
Since $bsu \in (\mathfrak{a}^{-1})^{\ell}$, $\pi(bsu) \in \pi(H_{J,K})[\mathfrak{a}]$. 
We can conclude that $\eta$ is surjective since 
\begin{align*}
\eta(\pi(bsu)) = \pi_{S}(bsu) = \pi_{S}(bsu) + \pi_{S}(au) = \pi_{S}(u). 
\end{align*}

Since $\pi(H_{J,K})[\mathfrak{a}] \simeq \pi_{S}(H_{J,K})[\mathfrak{a}\mathcal{O}_{S}]$ as $\mathcal{O}$-modules, \zcref{pid case} implies 
\begin{align*}
\# Z[\mathfrak{a}] 
= \# \pi(H_{J,K})[\mathfrak{a}]
= \# \pi_{S}(H_{J,K})[\mathfrak{a}\mathcal{O}_{S}]
= N_{\mathcal{O}_{S}}(\mathfrak{a}\mathcal{O}_{S})^{\dim_{K} H_{J,K}}
= N(\mathfrak{a})^{\dim_{K} H_{J,K}}
= N(\mathfrak{a})^{\dim Z}, 
\end{align*}
where $N_{\mathcal{O}_{S}}$ denotes the absolute norm on $\mathcal{O}_{S}$. 
\end{proof}

Our next goal is to establish a connection between the number of atoms below a layer $Z$ and the weight of an element not belonging to $Z$ (see \zcref{a(Z)}). 
The following lemma serves as a preparation for it. 

\begin{lemma}\label{sharing layer <=> parallel}
Suppose that $\alpha, \beta \in \mathcal{A}$. 
Then the following conditions are equivalent. 
\begin{enumerate}[(1)]
\item\label{sharing layer <=> parallel 1} 
$H_{\alpha_,K/\mathcal{O}}$ and $H_{\beta,K/\mathcal{O}}$ share a layer. 

\item\label{sharing layer <=> parallel 2} 
$H_{\alpha,K} = H_{\beta,K}$. 

\item\label{sharing layer <=> parallel 3} 
$\alpha$ is parallel to $\beta$. 
\end{enumerate}
\end{lemma}
\begin{proof}
It is clear that the statements (\ref{sharing layer <=> parallel 2}) and (\ref{sharing layer <=> parallel 3}) are equivalent by definition. 
Suppose $H_{\alpha,K/\mathcal{O}}$ and $H_{\beta,K/\mathcal{O}}$ share a layer $Y$. 
Then there exist $\pi(u) \in H_{\alpha,K/\mathcal{O}}$ and $\pi(v) \in H_{\beta,K/\mathcal{O}} $ such that 
\begin{align*}
Y = \pi(u) + \pi(H_{\alpha, K}) = \pi(v) + \pi(H_{\beta,K}). 
\end{align*}
By \zcref{KT Lem 4.3}, $H_{\alpha,K} = H_{\beta,K}$. 
Conversely, suppose $H_{\alpha,K} = H_{\beta,K}$. 
Then $H_{\alpha,K/\mathcal{O}}$ and $H_{\beta,K/\mathcal{O}}$ share the layer $\pi(H_{\alpha,K})$. 
\end{proof}

Given a layer $Z \in L$, define the \textbf{complement} $M(Z)$ of $Z$ by 
\begin{align*}
M(Z) \coloneqq Z \setminus \bigcup_{Y > Z}Y. 
\end{align*}
Recall $a(Z)$ denotes the number of atoms below $Z$.

\begin{lemma}\label{a(Z)}
Let $Z \in L$ and $\pi(u) \in Z$. 
Then the following statements hold. 
\begin{enumerate}[(1)]
\item\label{a(Z) 1} $a(Z) \leq p(\mathcal{A})-\wt(\pi(u)G)$. 
\item\label{a(Z) 2} $a(Z) = p(\mathcal{A})-\wt(\pi(u)G)$ if and only if $\pi(u) \in M(Z)$. 
\end{enumerate}
\end{lemma}
\begin{proof}
(\ref{a(Z) 1}) 
Let $D$ be the set of atoms $A$ of $L$ such that $A \leq Z$ and $R$ the set of parallel classes $C$ of $\mathcal{A}$ such that $\pi(u) \in H_{\alpha,K/\mathcal{O}}$ for some $\alpha \in C$. 
Then $\# D = a(Z)$ and $\#R = p(\mathcal{A}) - \wt(\pi(u)G)$. 

Let $A \in D$ and suppose that $A$ is a layer of both of $H_{\alpha,K/\mathcal{O}}$ and $H_{\beta,K/\mathcal{O}}$.  
Then \zcref{sharing layer <=> parallel} implies that $\alpha$ is parallel to $\beta$. 
Let $C_{\alpha}$ denote the parallel class containing $\alpha$. 
Then $C_{\alpha} \in R$ since $\alpha \in C_{\alpha}$ and $\pi(u) \in Z \subseteq A \subseteq H_{\alpha,K/\mathcal{O}}$. 
Therefore the map $f \colon D \to R$ defined by $f(A) \coloneqq C_{\alpha}$ is well-defined. 

Let us show that $f$ is injective. 
Suppose that $f(A) = f(B)$. 
When $A$ (resp. $B$) is a layer of $H_{\alpha,K/\mathcal{O}}$ (resp. $H_{\beta,K/\mathcal{O}}$), 
\begin{align*}
A = \pi(u) + \pi(H_{\alpha,K}) \qquad (\text{resp. } B = \pi(u) + \pi(H_{\beta,K}))
\end{align*}
and $C_{\alpha} = f(A) = f(B) = C_{\beta}$. 
By \zcref{sharing layer <=> parallel}, $H_{\alpha,K} = H_{\beta,K}$. 
Thus $A = B$. 
Therefore $f$ is injective and $a(Z) \leq p(\mathcal{A})-\wt(\pi(x)G)$. 

(\ref{a(Z) 2}) 
Suppose that the equality $a(Z) = p(\mathcal{A})-\wt(\pi(u)G)$ holds. 
Assume that $\pi(u) \not\in M(Z)$. 
Then there exists a layer $Y \in L(\mathcal{A}(K/\mathcal{O}))$ such that $Y > Z$ and $\pi(u) \in Y$. 
Then $a(Z) \leq a(Y)$ and $a(Y) \leq p(\mathcal{A})-\wt(\pi(u)G)$ by (\ref{a(Z) 1}). 
Therefore $a(Z) = a(Y)$. 
Namely, the sets of atoms of $L$ dominated by $Z$ and $Y$ coincide. 
For any $\alpha \in \mathcal{A}$, by \zcref{KT Prop 4.2}, $Z \subseteq H_{\alpha, K/\mathcal{O}}$ if and only if there exists an atom $A \subseteq H_{\alpha, K/\mathcal{O}}$ such that $Z \subseteq A$. 
Since the similar statement holds for $Y$ and the sets of atoms dominated by $Z$ and $Y$ coincide, $Z \subseteq H_{\alpha,K/\mathcal{O}}$ if and only if $Y \subseteq H_{\alpha,K/\mathcal{O}}$, which implies $H_{Z} = H_{Y}$ by \zcref{KT Lem 4.3}. 
Moreover, since both $Z$ and $Y$ contains $\pi(u)$, we can conclude that $Z = Y$. 
However, this contradicts to $Y < Z$. 
Hence $\pi(y) \in M(Z)$. 

Next, assume that the strict inequality $a(Z) < p(\mathcal{A})- \wt(\pi(u)G)$ holds. 
Then the map $f \colon D \to R$ defined in the proof of (\ref{a(Z) 1}) is not surjective. 
Take $C \in R \setminus f(D)$. 
Then there exists $\alpha \in C$ such that $\pi(u) \in H_{\alpha, K/\mathcal{O}}$ by the definition of $R$. 
Let $Y$ be a layer of $H_{J_{Z},K/\mathcal{O}} \cap H_{\alpha, K/\mathcal{O}}$ such that $\pi(u) \in Y$. 
Since $Y \subseteq H_{J_{Z},K/\mathcal{O}}$ and both $Z$ and $Y$ contains $\pi(u)$, \zcref{KT Prop 4.2} and \zcref{KT Lem 4.3} imply that $Y \subseteq Z$. 
Assume that $Y = Z$. 
Then $Z = Y \subseteq H_{\alpha, K/\mathcal{O}}$ and \zcref{KT Prop 4.2} imply that there exists an atom $A \subseteq H_{\alpha, K/\mathcal{O}}$ such that $Z \subseteq A$. 
Then $A \in D$ and $f(A) = C$, which is a contradiction. 
Therefore $Y \subsetneq Z$, or equivalently $Y > Z$. 
Thus $\pi(u) \not\in M(Z)$. 
\end{proof}

Finally, we give a proof of a key result (\zcref{layer decomposition}) stating that the cardinality of the $\mathfrak{a}$-torsion part of the complement $M(Z)$ can be expressed in terms of the poset $L[\mathfrak{a}]$ and the absolute norm $N(\mathfrak{a})$. 

\begin{lemma}\label{layer decomposition}
Let $\mathfrak{a} \in I(\mathcal{O})$ and $Z \in L[\mathfrak{a}]$. 
Then the following equalities hold. 
\begin{enumerate}[(1)]
\item\label{layer decomposition 1} $Z[\mathfrak{a}] = \bigsqcup_{\substack{Y \in L[\mathfrak{a}] \\ Y \geq Z}} M(Y)[\mathfrak{a}]$. 
In particular, $\left(\mathfrak{a}^{-1}/\mathcal{O}\right)^{\ell} = \bigsqcup_{Y \in L[\mathfrak{a}]} M(Y)[\mathfrak{a}]$. 
\item\label{layer decomposition 2}  $\# M(Z)[\mathfrak{a}] = \sum_{\substack{Y \in L[\mathfrak{a}] \\ Y \geq Z}}\mu(Z,Y)N(\mathfrak{a})^{\dim Y}$.
\end{enumerate}
\end{lemma}
\begin{proof}
(\ref{layer decomposition 1}) Taking the $\mathfrak{a}$-torsion parts of the both sides of 
$Z = \bigsqcup_{\substack{Y \in L \\ Y \geq Z}}M(Y)$
yields 
$Z[\mathfrak{a}] = \bigsqcup_{\substack{Y \in L \\ Y \geq Z}}M(Y)[\mathfrak{a}]$. 
For any $Y \in L \setminus L[\mathfrak{a}]$, $M(Y)[\mathfrak{a}] \subseteq Y[\mathfrak{a}] = \varnothing$ by \zcref{torsion lemma}. 
Therefore the desired equality holds.

(\ref{layer decomposition 2}) By \zcref{cardinality of Z[a]} and (\ref{layer decomposition 1}) , 
\begin{align*}
N(\mathfrak{a})^{\dim Z} = \sum_{\substack{Y \in L[\mathfrak{a}] \\ Y \geq Z}} \# M(Y)[\mathfrak{a}]. 
\end{align*}
By the M\"obius inversion formula, the desired equality holds. 

\end{proof}

\section{Proof of \zcref{main theorem}}\label{sec:proof}

Now, we are ready to prove the main theorem. 

\begin{proof}[Proof of \zcref{main theorem}]
(\ref{main theorem 1}) Since $\overline{\chi}_{\mathcal{A}}^{\mathrm{quasi}}(\mathfrak{a}, 0) = B_{0}(\mathfrak{a})$, the desired equality holds by the definitions of $B_{0}(\mathfrak{a})$ and $\chi_{\mathcal{A}}^{\mathrm{quasi}}(\mathfrak{a})$. 

(\ref{main theorem 2}) 
Let $[u]_{\mathfrak{a}} \in \left(\mathcal{O}/\mathfrak{a}\right)^{\ell}$ and put 
\begin{align*}
J([u]_{\mathfrak{a}}) \coloneqq \Set{\alpha \in \mathcal{A} | [u]_{\mathfrak{a}}\alpha = 0}. 
\end{align*}
Then 
\begin{align*}
\wt([u]_{\mathfrak{a}}G) = \# \Set{C \in \mathcal{P}(\mathcal{A}) | C \cap J([u]_{\mathfrak{a}}) = \varnothing}.  
\end{align*}
Therefore  
\begin{align*}
\Set{[u]_{\mathfrak{a}} \in \left(\mathcal{O}/\mathfrak{a}\right)^{\ell} | \wt([u]_{\mathfrak{a}}G) = p(\mathcal{A})-i} 
= \bigsqcup_{J} \Set{[u]_{\mathfrak{a}} \in \left(\mathcal{O}/\mathfrak{a}\right)^{\ell} | J([u]_{\mathfrak{a}}) = J}, 
\end{align*}
where $J$ runs over the subsets of $\mathcal{A}$ such that $\# \Set{C \in \mathcal{P}(\mathcal{A}) | C \cap J = \varnothing} = p(\mathcal{A})-i$. 
Put  
\begin{align*}
B_{J}(\mathfrak{a}) \coloneqq \#\Set{[u]_{\mathfrak{a}} \in \left(\mathcal{O}/\mathfrak{a}\right)^{\ell} | J([u]_{\mathfrak{a}}) = J}. 
\end{align*}
Then $B_{i}(\mathfrak{a}) = \sum_{J}B_{J}(\mathfrak{a})$. 
By the inclusion-exclusion principle, 
\begin{align*}
B_{J}(\mathfrak{a}) = \# \left(H_{J,\mathcal{O}/\mathfrak{a}} \setminus \bigcup_{J \subsetneq J^{\prime} \subseteq \mathcal{A}}H_{J^{\prime},\mathcal{O}/\mathfrak{a}}\right)
= \sum_{J \subseteq J^{\prime} \subseteq \mathcal{A}} (-1)^{\# J^{\prime} - \# J} \# H_{J^{\prime}, \mathcal{O}/\mathfrak{a}}. 
\end{align*}
By \cite[Lemma 3.2]{kuroda2024characteristic-ecaa}, each $\# H_{J^{\prime},\mathcal{O}/\mathfrak{a}}$ is a quasi-polynomial in $\mathfrak{a}$ of degree at most $\ell$ with period $\rho_{\mathcal{A}}$. 
Therefore $B_{J}(\mathfrak{a})$ and $B_{i}(\mathfrak{a})$ are quasi-polynomials in $\mathfrak{a}$ of degree at most $\ell$ with period $\rho_{\mathcal{A}}$.

(\ref{main theorem 3}) By (\ref{main theorem 2}), $\overline{\chi}_{\mathcal{A}}^{\mathrm{quasi}}(\mathfrak{a},x)$ is a quasi-polynomial in $\mathfrak{a}$ with period $\rho_{\mathcal{A}}$. 
By \zcref{charateristic quasi-polynomial} and (\ref{main theorem 1}), $\rho_{\mathcal{A}}$ is minimum. 

(\ref{main theorem 4}) Put $L \coloneqq L(\mathcal{A}(K/\mathcal{O}))$ and 
\begin{align*}
f(t,x) \coloneqq \sum_{Z \in L[\kappa]}\left(\sum_{\substack{Y \in L[\kappa] \\ Y \geq Z}}\mu(Z,Y)t^{\dim Y}\right)x^{a(Z)}. 
\end{align*}
It suffices to show that $\overline{\chi}_{\mathcal{A}}^{\mathrm{quasi}}(\mathfrak{a},x) = f(N(\mathfrak{a}),x)$ when $\mathfrak{a} + \rho_{\mathcal{A}} = \kappa$. 
Let us show this statement. 
\begin{align*}
\overline{\chi}_{\mathcal{A}}^{\mathrm{quasi}}(\mathfrak{a},x) 
&= \sum_{i=0}^{p(\mathcal{A})}B_{i}(\mathfrak{a})x^{i} 
= \sum_{[u]_{\mathfrak{a}} \in \left(\mathcal{O}/\mathfrak{a}\right)^{\ell}}x^{p(\mathcal{A})-\wt([u]_{\mathfrak{a}}G)} 
= \sum_{\pi(u) \in \left(\mathfrak{a}^{-1}/\mathcal{O}\right)^{\ell}}x^{p(\mathcal{A})-\wt(\pi(u)G)} \\
&= \sum_{Z \in L[\mathfrak{a}]}\sum_{\pi(u)\in M(Z)[\mathfrak{a}]}x^{p(\mathcal{A})-\wt(\pi(u)G)} 
= \sum_{Z \in L[\mathfrak{a}]}\sum_{\pi(u)\in M(Z)[\mathfrak{a}]}x^{a(Z)} \\
&= \sum_{Z \in L[\mathfrak{a}]}\# M(Z)[\mathfrak{a}]x^{a(Z)} 
= \sum_{Z \in L[\mathfrak{a}]}\left(\sum_{\substack{Y \in L[\mathfrak{a}] \\ Y \geq Z}}\mu(Z,Y) N(\mathfrak{a})^{\dim Y}\right)x^{a(Z)} \\
&= \sum_{Z \in L[\kappa]}\left(\sum_{\substack{Y \in L[\kappa] \\ Y \geq Z}}\mu(Z,Y) N(\mathfrak{a})^{\dim Y}\right)x^{a(Z)} 
= f(N(\mathfrak{a}), x). 
\end{align*}
Indeed, the first equality is the definition of $\overline{\chi}_{\mathcal{A}}^{\mathrm{quasi}}(\mathfrak{a},x)$. 
The second equality follows from the definition of $B_{i}(\mathfrak{a})$. 
The third equality follows from the isomorphism $\mathcal{O}/\mathfrak{a} \simeq \mathfrak{a}^{-1}/\mathcal{O}$. 
The forth equality follows from \zcref{layer decomposition}(1). 
The fifth equality follows from \zcref{a(Z)}(\ref{a(Z) 2}). 
The sixth equality holds since each term does not depend on $\pi(u)$. 
The seventh equality holds by \zcref{layer decomposition}(2). 
The eighth equality follows from $L[\mathfrak{a}] = L[\kappa]$ by \zcref{KT Cor4.1}. 
The last equality follows from the definition of $f(t,x)$. 

(\ref{main theorem 5}) By \zcref{KT Cor4.2}, the isomorphism $L[\mathcal{O}] \simeq L(\mathcal{A}(K))$ holds. 
Thus $\overline{\chi}_{\mathcal{A}}^{\mathcal{O}}(t,x) = \overline{\chi}_{\mathcal{A}(K)}(t,x)$ by (\ref{main theorem 4}) and the definition of $\overline{\chi}_{\mathcal{A}(K)}(t,x)$. 
\end{proof}

\section{Question about MacWilliams identity}\label{sec:question}

Let $C$ be a linear code in $\mathbb{F}_{p^{m}}^{n}$. 
The \textbf{dual} of $C$ is the linear code $C^{\perp}$ defined by 
\begin{align*}
C^{\perp} \coloneqq \Set{u \in \mathbb{F}_{p^{m}}^{n} | u \cdot v = 0 \text{ for all } v \in C}, 
\end{align*}
where $\cdot$ denotes the dot product. 
MacWilliams proved the following relation between the weight enumerators of $C$ and $C^{\perp}$. 

\begin{theorem}[MacWilliams identity {\cite{macwilliams1963theorem-tbstj}}]\label{MacWilliams}
Let $C$ be a linear code in $\mathbb{F}_{p^{m}}^{n}$. 
Then 
\begin{align*}
W_{C}(x,y) = \dfrac{1}{\# C^{\perp}}W_{C^{\perp}}(x+(p^{m}-1)y, x-y). 
\end{align*}
\end{theorem}

For a finite abelian group $G$, let $\widehat{G} \coloneqq \Hom(G,\mathbb{Q}/\mathbb{Z})$ be the character group of $G$. 
For a finite ring $R$, $\widehat{R}$ denotes the character group of the additive group of $R$. 
The group $\widehat{R}$ is naturally regarded as the right $R$-module by the action $(\chi r)(a) \coloneqq \chi(ra)$, where $\chi \in \widehat{R}$ and $r,a \in R$. 
One of key points in the character theoretic proof of the MacWilliams identity is the fact that $\widehat{\mathbb{F}}_{p^{m}}$ is isomorphic to $\mathbb{F}_{p^{m}}$ as vector spaces over $\mathbb{F}_{p^{m}}$ (See \cite[Example 2]{wood2012applications-potsortart}, for example). 

For a finite ring $R$, if $\widehat{R}$ is isomorphic to $R$ as right $R$-modules, then a similar proof works and MacWilliams identity is generalized for such rings. 
According to \cite[Theorem 3.10]{wood1999duality-ajom}, $\widehat{R} \simeq R$ as right $R$-modules if and only if $R$ is a \textbf{Frobenius ring} when $R$ is a finite ring. 

We call a left $R$-submodule $C$ in $R^{n}$ a \textbf{left linear code}. 
The \textbf{right annihilator} $r(C)$ is defined by 
\begin{align*}
r(C) \coloneqq \Set{u \in R^{n} | u \cdot v = 0 \text{ for all } v \in C}, 
\end{align*}
where $\cdot$ denotes the dot product. 
If $R$ is commutative, we also use the notation $C^{\perp}$ for $r(C)$. 
Let $W_{C}(x,y)$ and $W_{r(C)}(x,y)$ denote the weight enumerator of $C$ and $r(C)$ with respect to the Hamming weight. 
Wood proved the following theorem. 

\begin{theorem}[{\cite[Theorem 8.1]{wood1999duality-ajom}}, {\cite[Theorem 23]{wood2012applications-potsortart}}]\label{Wood}
Let $R$ be a finite Frobenius ring and $C \subseteq R^{n}$ a left linear code.  
Then
\begin{align*}
W_{C}(x,y) = \dfrac{1}{\# r(C)}W_{r(C)}(x+(\# R -1)y, x-y). 
\end{align*}
\end{theorem}

We now return to our setting.
Let $\mathcal{O}$ be a residually finite Dedekind domain and $\mathcal{A}$ a finite subset of nonzero vectors in $\mathcal{O}^{\ell}$. 
For each nonzero ideal $\mathfrak{a}$ of $\mathcal{O}$, let $G_{\mathfrak{a}}$ be a map from $(\mathcal{O}/\mathfrak{a})^{\ell}$ to $(\mathcal{O}/\mathfrak{a})^{\mathcal{A}}$ defined by $[u]_{\mathfrak{a}} \mapsto ([u]_{\mathfrak{a}}\alpha)_{\alpha \in \mathcal{A}}$. 
Then $\Image G_{\mathfrak{a}}$ and $(\Image G_{\mathfrak{a}})^{\perp}$ are linear codes in $(\mathcal{O}/\mathfrak{a})^{\mathcal{A}}$.

By the Chinese remainder theorem, every residue ring $\mathcal{O}/\mathfrak{a}$ is isomorphic to the direct sum of rings of the form $\mathcal{O}/\mathfrak{p}^{i}$, where $\mathfrak{p}$ ranges over prime ideals dividing $\mathfrak{a}$.  
The ring $\mathcal{O}/\mathfrak{p}^{i}$ is a finite chain ring, that is, a finite ring whose ideals form a chain under inclusion. 
According to Wood \cite[Example 6]{wood2012applications-potsortart}, the finite direct sum of finite chain rings is a finite Frobenius ring. 
Thus, the residue ring $\mathcal{O}/\mathfrak{a}$ is a finite Frobenius ring and we have the following corollary.

\begin{corollary}\label{ImGa}
If every parallel class of $\mathcal{A}$ is a singleton, then 
\begin{align*}
W_{\Image G_{\mathfrak{a}}}(x,y) = \dfrac{1}{\# (\Image G_{\mathfrak{a}})^{\perp}}W_{(\Image G_{\mathfrak{a}})^{\perp}}(x+(N(\mathfrak{a})-1)y, x-y). 
\end{align*}
\end{corollary}

Greene \cite{greene1976weight-siam} gave another proof of the MacWilliams identity (\zcref{MacWilliams}) by using \zcref{Greene} and the duality of the Tutte polynomial. 
It is a natural question to ask whether we can prove \zcref{ImGa} by the same idea. 
Namely, are there the ``dual notions" $\mathcal{A}^{\ast}$ and $G^{\ast}_{\mathfrak{a}}$ such that $W_{\Image G^{\ast}_{\mathfrak{a}}}(x,y) = W_{(\Image G_{\mathfrak{a}})^{\perp}}(x,y)$?

Fink and Moci \cite{fink2016matroids-jotems} introduced ``matroids over a ring".
Every finite subset $\mathcal{A}$ of $\Gamma \coloneqq \mathcal{O}^{\ell}$ determines a matroid over $\mathcal{O}$ and we can take its dual. 
The dual matroid is also representable. 
Namely, it is determined by a finite subset $\mathcal{A}^{\ast}$ in a certain finitely generated $\mathcal{O}$-module $\Gamma^{\ast}$. 
We can define a map $G^{\ast}_{\mathfrak{a}}$ by a similar way. 

\begin{question}
Let $\mathcal{A} \subseteq \mathcal{O}^{\ell}$ be a finite subset and $\mathcal{A}^{\ast}$ its dual in the sense of Fink and Moci. 
Does the equality $W_{\Image G^{\ast}_{\mathfrak{a}}}(x,y) = W_{(\Image G_{\mathfrak{a}})^{\perp}}(x,y)$ hold?
\end{question}

We give an easy example for the question above. 
\begin{example}
Consider the case $\mathcal{O} = \mathbb{Z}$. 
Let $\Gamma = \mathbb{Z}^{2}$ and $\mathcal{A} = \{\alpha_{1}, \alpha_{2}\} \subseteq \Gamma$, where 
\begin{align*}
\alpha_{1} = \begin{pmatrix}
2 \\ 0
\end{pmatrix}, \qquad
\alpha_{2} = \begin{pmatrix}
0 \\ 1
\end{pmatrix}. 
\end{align*}
Then $\Gamma^{\ast} = \mathbb{Z}/2\mathbb{Z}$ and $\mathcal{A}^{\ast}=\{\alpha_{1}^{\ast}, \alpha_{2}^{\ast}\}$, where $\alpha_{1}^{\ast} = 1, \alpha_{2}^{\ast} = 0$. 

For a positive integer $q$, we define a map $G_{q} \colon \Hom(\Gamma, \mathbb{Z}/q\mathbb{Z}) \to (\mathbb{Z}/q\mathbb{Z})^{2}$ by $\phi \mapsto (\phi(\alpha_{1}), \phi(\alpha_{2}))$. 
Then $\Image G_{q} = \langle (2,0), (0,1) \rangle$ and 
\begin{align*}
(\Image G_{q})^{\perp} = 
\begin{cases}
\{(0,0)\} & \text{ if $q$ is odd}; \\
\{(0,0), (\frac{q}{2},0)\} & \text{ if $q$ is even}. 
\end{cases}
\end{align*}
Therefore 
\begin{align*}
W_{\Image G_{q}}(x,y) &= 
\begin{cases}
x^{2} + 2(q-1)xy + (q-1)^{2}y^{2} & \text{ if $q$ is odd}; \\
x^{2} + (\frac{3}{2}q-2)xy + (\frac{q}{2}-1)(q-1)y^{2} & \text{ if $q$ is even}, 
\end{cases} \\
W_{(\Image G_{q})^{\perp}}(x,y) &= 
\begin{cases}
x^{2} & \text{ if $q$ is odd}; \\
x^{2}+xy & \text{ if $q$ is even}. 
\end{cases}
\end{align*}
We can verify that the MacWilliams identity holds. 

Define the map $G^{\ast}_{\mathfrak{a}} \colon \Hom(\Gamma^{\ast},\mathbb{Z}/q\mathbb{Z}) \to (\mathbb{Z}/q\mathbb{Z})^{2}$ by $\phi \mapsto (\phi(\alpha^{\ast}_{1}), \phi(\alpha^{\ast}_{2}))$.  
Then we can compute $\Image G^{\ast}_{q} = (\Image G_{q})^{\perp}$. 
Thus $W_{G^{\ast}_{q}}(x,y) = W_{(\Image G_{q})^{\perp}}(x,y)$. 
\end{example}

\section*{Acknowledgments}
S. T. was supported by JSPS KAKENHI, Grant Number JP22K13885 and JP23H00081. 
This work was supported by Institute of Mathematics for Industry, Joint Usage/Research Center in Kyushu University. (FY2023 Workshop(II) ``Invariants for discrete mathematics'' (2023a002) and FY2024 Workshop(II) ``Polynomial invariants related to error-correcting codes and hyperplane arrangements'' (2024a018)).

\bibliographystyle{amsplain}
\bibliography{bibfile}

\end{document}